\documentclass[10 pt]{elsarticle}
\usepackage{helvet}

\usepackage{amsmath}
\usepackage{amssymb}

\newtheorem{Def}{Definition}
\newtheorem{Prop}{Proposition}
\newtheorem{Thm}{Theorem}

\newtheorem{Cor}{Corollary}
\newproof{proof}{Proof}

\newenvironment{keywords}
{ \textbf{Keywords:} \begin{itshape} }
{\end{itshape} }

\begin{document}

\pagestyle{plain}
\setcounter{page}{1}
\pagenumbering{arabic}

\begin{frontmatter}

\title{Stable adiabatic times for Markov chains}

\author[Math]{Kyle Bradford\corref{cor1}}

\author[Math]{Yevgeniy Kovchegov}

\author[EECS]{Thinh Nguyen}

\address[Math]{Department of Mathematics \\ 
Oregon State University \\ 
368 Kidder Hall \\
Corvallis, OR 97331 á 541-737-4686, USA}

\address[EECS]{School of Electrical Engineering and Computer Science \\
Oregon State University \\
1148 Kelley Engineering Center \\
Corvallis, OR 97331-5501, USA}

\ead{bradfork@science.oregonstate.edu (Kyle Bradford), kovchegy@math.oregonstate.edu (Yevgeniy Kovchegov), thinhq@eecs.oregonstate.edu}

\cortext[cor1]{Corresponding author.}

\begin{abstract}
\label{sec: abstract}
In this paper we continue our work on {\bf adiabatic time} of time-inhomogeneous Markov chains first introduced in \cite{kovchegov2010note}  and \cite{bradford2011adiabatic}.  Our study is an analog to the well-known Quantum Adiabatic (QA) theorem which characterizes the quantum adiabatic time for the evolution of a quantum system as a result of applying of a series of Hamilton operators, each is a linear combination of two given initial and final Hamilton operators, i.e.
$\mathbf{H}(s) = (1-s)\mathbf{H_0} + s\mathbf{H_1}$. Informally, the quantum adiabatic time of a quantum system specifies the speed at which the Hamiltonian operators changes so that the ground state of the system at any time $s$ will always remain $\epsilon$-close to that induced by the Hamilton operator $\mathbf{H}(s)$ at time $s$. Analogously, we derive a sufficient condition for the stable adiabatic time of a time-inhomogeneous Markov evolution specified by applying a series of transition probability matrices, each is a linear combination of two given irreducible and aperiodic transition probability matrices, i.e., $\mathbf{P_{t}} = (1-t)\mathbf{P_{0}} + t\mathbf{P_{1}}$.  In particular we show that the stable adiabatic time $t_{sad}(\mathbf{P_{0}}, \mathbf{P_{1}}, \epsilon) = O \left( t_{mix}^{4}(\epsilon \slash 2) \slash \epsilon^{3} \right), $  where $t_{mix}$ denotes the maximum mixing time over all $\mathbf{P_{t}}$ for $0 \leq t \leq 1$.

\end{abstract}

\end{frontmatter}
\thispagestyle{empty}


\section{\sc {Introduction}}
\label{sec: introduction}

In this paper we study the stability of time-inhomogeneous Markov chains via the notion of stable adiabatic time, an extension of the adiabatic time first introduced in \cite{kovchegov2010note} and \cite{bradford2011adiabatic}.  Our study is motivated in part by the well-known Quantum Adiabatic (QA) theorem which characterizes the quantum adiabatic time for the evolution of a quantum system as a result of applying of a series of Hamiltonian operators, each is a linear combination of two pre-specified initial and final Hamilton operators, i.e., $\mathbf{H}(s) = (1-s)\mathbf{H_0} + s\mathbf{H_1}$.  Quantum adiabatic time of a quantum system, to be discussed in detail shortly, specifies the rate at which Hamiltonian operators change so that the ground state of the system at any time $s$ will always remain $\epsilon$-close to that induced by the Hamilton operator $\mathbf{H}(s)$ at time $s$.   The first Quantum Adiabatic theorem was stated in the 1920s by M. Born and V.A. Fock \cite{fock2004selected}, and have been subsequently studied in \cite{kato1950on}  among others.  Recently, the quantum adiabatic time plays an important role in the development of quantum adiabatic computing.  Specifically, quantum adiabatic algorithms are constructed as a sequence of Hamilton operators applied to a quantum system in such a way that drives the system to the desirable state or output, see for example \cite{krovi2010adiabatic}.  Thus, the quantum adiabatic time is a natural choice for characterizing the running times of adiabatic quantum algorithms.

%We recognize the difference between the quantum systems and the time-inhomogeneous Markov chains discussed in this paper, but we note that the Quantum Adiabatic theorem inspired our research in this area.

We analogously derive a sufficient condition for the stable adiabatic time of a time-inhomogeneous Markov evolution specified by applying a series of transition probability matrices, each is a linear combination of two given irreducible and aperiodic transition probability matrices, i.e., $\mathbf{P_{t}} = (1-t)\mathbf{P_{0}} + t\mathbf{P_{1}}$.  In particular we show that the stable adiabatic time $t_{sad}(\mathbf{P_{0}}, \mathbf{P_{1}}, \epsilon) = O \left( t_{mix}^{4}(\epsilon \slash 2) \slash \epsilon^{3} \right), $  where $t_{mix}$ denotes the maximum mixing time induced over all the transition probability matrices during the evolution.  We note the stable adiabatic time for time-inhomogeneous Markov has recently found practical applications in network design.  We refer to the recent work of Rajagoplan et al. \cite{rajagopalan2009network} where adiabatic time were used to design optimal medium access protocols in wireless networks.  Recently, the time-inhomogeneous evolution $\mathbf{P_{t}} = (1-t)\mathbf{P_{0}} + t\mathbf{P_{1}}$  has also been used to describe the performance of queueing models \cite {zacharias2011adiabatic} for networks.  Specifically,  in this setting, the arrival rate of packet at the queue is assumed to be unknown and is estimated progressively. Appropriate sending rate is then determined based on this estimation. As a result, $\mathbf{P_{t}}$ describes the a queuing policy (or sending rate) which varies with time based on the new statistics.  The adiabatic time is then used to characterize the performance of the queuing model under uncertainty due error in estimation. To motivate our work, we now provide a short overview of the Quantum Adiabatic theorem as discussed in \cite{ambainis2006elementary}.

\subsection{\sc {Quantum Adiabatic Theorem}}
\label{subsec: QAT}
\

Let $\mathbf{H_0}$ and $\mathbf{H_1}$ be two given Hamiltonian operators.  Let $T> 0$  be a positive integer.  For $t \in [0,T]$, denote $\mathbf{H}(s) = (1-s)\mathbf{H_0} + s\mathbf{H_1}$, then $\mathbf{H}(s)$ is also a Hamiltonian operator dependent on a time parameter $s = t \slash T$.  In general, a Hamiltonian operator described above, is not required to have a finite number of physical pure states. In this paper we only consider transition probability matrices of finite dimension, which are analogous the Hamiltonian operators with a finite number of $n$  physical pure states.

The ground state in quantum mechanics refers to the lowest-energy state.  The quantum adiabatic theorem concerns one eigenstate of the energy function, the ground state.  Here we denote $\mathbf{\Phi}(s)$  as the ground state of $\mathbf{H}(s)$  and we let $\gamma(s)$  be the eigenvalue associated with it.  For a given $T > 0$, when we say that we apply the adiabatic evolution given by $\mathbf{H}$  and $\mathbf{\Phi}$  for time $T$ we mean to initialize the system in the state $\mathbf{\Phi}(0)$  and then apply the continuously varying Hamiltonian $\mathbf{H}(t \slash T)$  for time $t \in [0,T]$.

Given $\epsilon>0$  the quantum adiabatic theorem informally says that if, by selecting a large enough value of $T$, we assume that the change in the Hamiltonian happens slowly enough, then when we apply the adiabatic evolution given by $\mathbf{H}(s)$  and $\mathbf{\Phi}$  for time $T$  we will be in an $\epsilon$-ball around $\mathbf{\Phi}(1)$ with respect to the $l^{2}(\mathbb{C}^{n})$-norm.  This leads us to the following definition.

\begin{Def} \label{def: three}
Given $\epsilon>0$ the \underline{quantum adiabatic time}, denoted as $t_{qad}(\mathbf{H}, \mathbf{\Phi}, \epsilon)$, is equal to the smallest positive time, T,  required to make the application of the adiabatic evolution given by $\mathbf{H}$  and $\mathbf{\Phi}$  for time $T$  arrive in an $\epsilon$-ball around $\mathbf{\Phi}(1)$ with respect to the $l^{2}(\mathbb{C}^{n})$-norm.
\end{Def}

The quantum adiabatic theorem gives a sufficient condition for the quantum adiabatic time.  The version of quantum adiabatic theorem as proved in \cite{ambainis2006elementary} motivated our work in \cite{kovchegov2010note} and \cite{bradford2011adiabatic}.  Supposing that all eigenvalues of $\mathbf{H}(s)$  are either smaller than $\gamma(s) - \Delta$  or larger than $\gamma(s) + \Delta$ (i.e. there is a spectral gap of $\Delta$  around $W_{1}(s)$), then  \cite{ambainis2006elementary} tells us that $$t_{qad}(\mathbf{H}, \mathbf{\Phi}, \epsilon) = O \left( \frac{1}{\epsilon^{2} \Delta^{4}} \right).$$

In \cite{fock2004selected}  the authors suggest that the quantum adiabatic theorem can be described as follows: for an infinitely slow change of the system, i.e., at an infinitely large value of $T$,  the probability of a quantum state changing energy levels remains infinitely small, even for finite values of $s = t \slash T$  so that $s \in [0,1]$. 

While the result in \cite{ambainis2006elementary} was used to motivate the work in \cite{kovchegov2010note}  and \cite{bradford2011adiabatic}, the latter notion in \cite{fock2004selected} is what motivated the work in this paper.  We describe an analogue to this quantum system in the context of the classical time-inhomogeneous Markov chains. 

\subsection{\sc {Main Result}
\label{subsec: compare}}
\

Markov chains over a finite number of $n$  states are $l^{1}(\mathbb{R}_{+}^{n})$-norm preserving processes.  We denote $\| \cdot \|_{1}$  as the $l^{1}(\mathbb{R}^{n})$ norm throughout the paper.  We assume that the reader has a prior understanding of basic types of Markov chains, such as irreducible, aperiodic, time-homogeneous and time-inhomogeneous Markov chains (see \cite{isaacson1976markov}, \cite{karlin1975first} and \cite{levin2009markov}). Throughout the paper we focus on discrete-time Markov chains and use $\| \cdot \|_{TV}$  to denote the total variation norm.
We begin with the definition of mixing time.

\begin{Def} \label{def: one}
For $\epsilon > 0$  the \underline{mixing time}  of a time-homogeneous, irreducible and aperiodic Markov chain governed by a probability transition matrix $\mathbf{P}$, which has unique stationary distribution $\mathbf{\pi}$,  is defined as:
\begin{equation} \label{eq: one}
t_{mix}(\mathbf{P}, \epsilon) = \inf \{ T \in \mathbb{N} : \| \mathbf{\nu} \mathbf{P}^{T} - \mathbf{\pi} \|_{TV} \leq \epsilon \}
 \end{equation}

\noindent over all distributions $\mathbf{\nu}$. \\
\end{Def}

 Mixing times of time-homogeneous irreducible and aperiodic Markov chains have been well studied, see for example \cite{aldous2002reversible} and \cite{levin2009markov}. Mixing time characterizes how fast a chain converges to its stationary distribution.  It is particularly important for bounding the running times of many randomized algorithms, for example the simulated annealing algorithm and metropolis-hasting algorithm as explained in \cite{ross2006simulation}. There have also been recent studies on mixing times for time-inhomogeneous Markov chain \cite{saloff2006convergence}, \cite{saloff2009merging} and \cite{saloff2011merging} which is more closely related to our work.  For example, Saloff-Coste and Z$\acute{\mathrm{u}}\tilde{\mathrm{n}}$iga  \cite{saloff2006convergence} consider spectral bounds of the mixing time for time-inhomogeneous Markov chains on a finite state space when each step in transition corresponds to an ergodic Markov kernel with the same stationary measure.  In their subsequent work \cite{saloff2009merging},  Saloff-Coste and Z$\acute{\mathrm{u}}\tilde{\mathrm{n}}$iga employed spectral techniques to obtain asymptotic behavior of time-inhomogeneous Markov chains.  In this work, the concept of c-stability is introduced, which is an abstraction of requiring all Markov kernels to have the same stationary measure, and bounds for this kind of stability are obtained.

Our work, on the other hand, studies a specific class of discrete-time, time-inhomogeneous Markov chains that was first constructed  in \cite{kovchegov2010note}.  Specifically, we consider the probability transition matrices for two discrete-time, time-homogeneous, irreducible and aperiodic Markov chains over $n$  states.  We denote these matrices as $\mathbf{P_{0}}$  and $\mathbf{P_{1}}$  throughout the paper and call these the initial and the final transition matrices respectively.  We define for $t \in [0,1]$  a class of probability transition matrices $\{ \mathbf{P_{t}} \}_{t \in [0,1]}$  such that $$ \mathbf{P_{t}} = (1 - t)\mathbf{P_{0}} + t\mathbf{P_{1}}. $$

For $t \in [0,1]$  we also define $\mathbf{\pi_{t}}$  to be the stationary distribution of $\mathbf{P_{t}}$.  Given $T \in \mathbb{N}$, the specific time-inhomogeneous Markov chain being considered in our paper is the one such that the probability transition matrix at time $k$  is $\mathbf{P_{\frac{k}{T}}}$  for $0 \leq k \leq T$.  We consider the class of all time-inhomogeneous Markov chains of this type over all $T \in \mathbb{N}$.  We will say that any Markov chain in this class is governed by an adiabatic evolution between $\mathbf{P_{0}}$  and $\mathbf{P_{1}}$.

Stability of these kinds of Markov chains were described in \cite{kovchegov2010note} and \cite{bradford2011adiabatic} using the notion of adiabatic time defined as follows:

\begin{Def} \label{def: two}
For $\epsilon > 0$  the \underline{adiabatic time}  of a time-inhomogeneous, discrete-time Markov chain governed by an adiabatic evolution between $\mathbf{P_{0}}$  and $\mathbf{P_{1}}$,  is defined as:
\begin{equation}
t_{ad}( \mathbf{P_{0}}, \mathbf{P_{1}}, \epsilon) = \inf \{ T^{*} \in \mathbb{N} : \max_{\mathbf{\nu}} \| \mathbf{\nu} \mathbf{P_{0}} \mathbf{P_{\frac{1}{T}}} \mathbf{P_{\frac{2}{T}}} \cdots \mathbf{P_{1}} - \mathbf{\pi_{1}} \|_{TV} \leq \epsilon \text{ for } T \in \mathbb{N}, T \geq T^{*} \},
\end{equation}

\noindent where $\mathbf{\nu}$  is a probability distribution.
\end{Def}

Next, we recall a result from \cite{kovchegov2010note}  that compares the adiabatic time  of a time-inhomogeneous Markov chain and the mixing time of the time-homogeneous Markov chain governed by the final transition matrix.

\begin{Thm} \label{th: one}
\

\noindent Given a time-inhomogeneous, discrete-time Markov chain governed by an adiabatic evolution between the two irreducible and aperiodic $\mathbf{P_{0}}$  and $\mathbf{P_{1}}$, for $\epsilon > 0$

\begin{equation} \label{eq: three}
t_{ad}(\mathbf{P_{0}}, \mathbf{P_{1}}, \epsilon) = O \left( \frac{t_{mix}^{2}(\mathbf{P_{1}}, \epsilon \slash 2)}{\epsilon} \right).
\end{equation}
\end{Thm}

We showed that this bound is tight in \cite{bradford2011adiabatic}  by finding a pair of matrices, $\mathbf{P_{0}}$  and $\mathbf{P_{1}}$, with the following property:  as $\epsilon \rightarrow 0$, there exists a positive constant $C$  such that $$t_{ad}(\mathbf{P_{0}}, \mathbf{P_{1}}, \epsilon) = \frac{C t_{mix}^{2}(\mathbf{P_{1}}, \epsilon \slash 2)}{\epsilon}.$$

In the following Proposition, we provide an upper bound on the adiabatic time using the square of the mixing time.  Although this is a minor improvement, it is necessary for our main result.  The proof of this Proposition is given in Section~\ref{sec: proofs}.

\begin{Prop} \label{pr: one}
Given a time-inhomogeneous, discrete-time Markov chain governed by an adiabatic evolution between the two irreducible and aperiodic $\mathbf{P_{0}}$  and $\mathbf{P_{1}}$, for $\epsilon > 0$
\begin{equation} \label{eq: four}
t_{ad}(\mathbf{P_{0}}, \mathbf{P_{1}}, \epsilon) \leq \frac{2 t_{mix}^{2}(\mathbf{P_{1}}, \epsilon \slash 2)}{\epsilon}.
\end{equation}
\end{Prop}

%The results in \cite{kovchegov2010note} were expanded upon in \cite{bradford2011adiabatic} for a more general setting, also showing the same order of $t_{mix}^{2}(\mathbf{P_{1}}, \epsilon) \slash \epsilon$  for the sufficient bound of the adiabatic time, and providing applications in statistical mechanics.  The work in \cite{kovchegov2010note}  and \cite{bradford2011adiabatic} make an analogue to the quantum process described above; one that considers time-inhomogeneous Markov processes.  The adiabatic time outlined in Definition \ref{def: two}  is the appropriate discrete-time analogue of the continuous-time quantum adiabatic time outlined in Definition \ref{def: three}.  The ground state in quantum mechanics corresponds to the stationary state of a Markov process.

Definition \ref{def: two}  suggests that for $\epsilon > 0$ and any $T \geq t_{ad}(\mathbf{P_{0}}, \mathbf{P_{1}}, \epsilon)$  any probability distribution will evolve under consecutive applications of $\mathbf{P_{\frac{k}{T}}}$  to an $\epsilon$-ball around $\mathbf{\pi_{1}}$ in the space of probability distributions with respect to the total variation norm.  We desire a stronger notion of stability in this paper to match the description of the quantum adiabatic theorem mentioned in \cite{fock2004selected}.   We want to select $T$  large enough so that starting at $\mathbf{\pi_{0}}$, the distribution will evolve under consecutive applications of $\mathbf{P_{\frac{k}{T}}}$ within an $\epsilon$-corridor of $\mathbf{\pi_{\frac{k}{T}}}$  for $1 \leq k \leq T$.  This leads us to the following definition.

\begin{Def} \label{def: four}
For $\epsilon > 0$ the \underline{stable adiabatic time}  of a time-inhomogeneous, discrete-time Markov chain governed by an adiabatic evolution between the irreducible and aperiodic $\mathbf{P_{0}}$  and $\mathbf{P_{1}}$, written as $t_{sad}(\mathbf{P_{0}}, \mathbf{P_{1}}, \epsilon)$, is defined as follows:
\begin{equation} \label{eq: six}
t_{sad}(\mathbf{P_{0}}, \mathbf{P_{1}}, \epsilon) = \inf \{ T \in \mathbb{N} : \| \mathbf{\pi_{0}} \mathbf{P_{\frac{1}{T}}} \cdots \mathbf{P_{\frac{k}{T}}} - \mathbf{\pi_{\frac{k}{T}}} \|_{TV} < \epsilon \text{ for } 1 \leq k \leq T \}.
\end{equation}
\end{Def}

The main goal of this paper is finding a bound for the stable adiabatic time with respect to the maximum mixing time over all the transition probability matrices.  For $\epsilon > 0$ we let $$t_{mix}(\epsilon) = \sup_{s \in [0,1]}  \{ t_{mix}(\mathbf{P_{s}}, \epsilon) \}$$

\noindent and we seek our bound in terms of this $t_{mix}(\epsilon)$. % This question was listed as an open problem by two of the authors in \cite{bradford2011adiabatic}. 
 
 We divide our result into two main theorems to highlight the nature of this bound.  Our first theorem gives us insight into the nature of the stable adiabatic time. Its proof is given in Section~\ref{sec: proofs}.

\begin{Thm} \label{th: three}
Given a time-inhomogeneous, discrete-time Markov chain governed by an adiabatic evolution between the irreducible and aperiodic $\mathbf{P_{0}}$  and $\mathbf{P_{1}}$  and given $\delta \in ( 0, 1 ]$, for any $\epsilon > 0$,  $$ \| \mathbf{\pi_{0}} \mathbf{P_{\frac{1}{T}}} \cdots \mathbf{P_{\frac{k}{T}}} - \mathbf{\pi_{\frac{k}{T}}} \|_{TV} \leq \epsilon  $$

\noindent for $$ T \geq \frac{2 t_{mix}^{2}(\epsilon \slash 2)}{\epsilon \delta}, $$

\noindent and $ \delta \leq k \slash T \leq 1$.
\end{Thm}
%
%Our remaining goal is to find a value of $\delta$  so that for large enough $T$, we have $$ \| \mathbf{\pi_{0}} \mathbf{P_{\frac{1}{T}}} \cdots \mathbf{P_{\frac{k}{T}}} - \mathbf{\pi_{\frac{k}{T}}} \|_{TV} \leq \epsilon  $$
%
%\noindent for $ k \slash T \leq \delta$.

We now state our main result.  The proof is given in Section \ref{sec: result}.

\begin{Thm} \label{th: four}
Given a time-inhomogeneous, discrete-time Markov chain governed by an adiabatic evolution between two time-homogeneous, discrete-time, $n$-state, irreducible and aperiodic Markov chains with probability transition matrices $\mathbf{P_{0}}$  and $\mathbf{P_{1}}$, for any $\epsilon > 0$ ,

\begin{equation} \label{eq: seven}
t_{sad}(\mathbf{P_{0}}, \mathbf{P_{1}}, \epsilon) = O \left( \frac{t_{mix}^{4}(\epsilon \slash 2)}{\epsilon^{3}} \right).
\end{equation}

\end{Thm}

The remaining sections of the paper are organized as follows:  in Section~\ref{sec: TT}  we state the necessary tools for the proof of our main theorem, in Section~\ref{sec: result}  we prove our main theorem, and Section~\ref{sec: proofs}  is dedicated to proofs. 

\section{\sc {Preliminaries}}
\label{sec: TT}
We begin this section with a result on the stability of time-homogeneous Markov chains.  We will find a lower bound for the mixing time of a time-homogeneous, discrete-time, irreducible and aperiodic Markov chain governed by the probability transition matrix $\mathbf{P}$  in terms of the inverse of the smallest nonzero singular value of $\mathbb{I} - \mathbf{P}$.  There are similar results in \cite{levin2009markov}, where the lower bound for the mixing time of a time-homogeneous, discrete-time, irreducible, aperiodic and reversible Markov chain governed by the probability transition matrix $\mathbf{P}$  is found in terms of the inverse smallest  nonzero eigenvalue of $\mathbb{I} - \mathbf{P}$, or rather in terms of the relaxation time for $\mathbf{P}$.  One should note that our work is not limited to reversible Markov chains.  Our work applies to a much larger class of Markov chains.  The proof of the following Proposition is in Section~\ref{sec: proofs}.

\begin{Prop} \label{pr: two}
\noindent For a time-homogeneous, discrete-time, $n$-state, irreducible and aperiodic Markov chain, if we are given $\epsilon > 0$, then if $\sigma$ is the smallest nonzero singular value of $\mathbb{I} - \mathbf{P}$, 
\begin{equation} \label{eq: eight}
\frac{1 - 2 \sqrt{n} \epsilon}{\sigma} \leq t_{mix}(\mathbf{P}, \epsilon).
\end{equation} 
\end{Prop}

This Proposition will be vital for proving Theorem \ref{th: four} and it gives us some intuition about the mixing time.  There have been many results bounding the relaxation time for reversible Markov chains on weighted graphs for example conductance bounds and weighted path upper bounds.  In both \cite{aldous2002reversible}  and \cite{bremaud2010markov}  the authors introduce the necessary spectral structure to find these bounds.  They also define a Dirichlet form to help derive the well-known Rayleigh Theorem and the Perron-Frobenius Theorem, which also describe bounds on the relaxation time.  Our work, however, does not employ these techniques directly. 

We now find a bound of $\| \mathbf{\pi_{0}} \mathbf{P_{\frac{1}{T}}} \cdots \mathbf{P_{\frac{k}{T}}} - \mathbf{\pi_{\frac{k}{T}}} \|_{TV}$  in terms of $\| \mathbf{\pi_{0}} - \mathbf{\pi_{\frac{k}{T}}}  \|_{TV}$.  We devote the following Proposition to this endeavor and its proof is in Section~\ref{sec: proofs}.

\begin{Prop} \label{pr: three}
For $1 \leq k \leq T$  

\begin{equation} \label{eq: nine}
\| \mathbf{\pi_{0}} \mathbf{P_{\frac{1}{T}}} \cdots \mathbf{P_{\frac{k}{T}}} - \mathbf{\pi_{\frac{k}{T}}} \|_{TV} \leq \| \mathbf{\pi_{\frac{k}{T}}} - \mathbf{\pi_{0}} \|_{TV} + \frac{(k+1)^{2}}{2 T}.
\end{equation}
\end{Prop}

Now we can use the continuity of $\mathbf{\pi_{s}}$  at $s=0$  to find an appropriate bound for $\| \mathbf{\pi_{0}} \mathbf{P_{\frac{1}{T}}} \cdots \mathbf{P_{\frac{k}{T}}} - \pi_{\frac{k}{T}} \|_{TV} $  for $0 \leq k \slash T \leq \delta$.  We devote the following Proposition to the discovery of how $\mathbf{\pi_{s}}$  is continuous at $s = 0$.  The spectral structure of $\mathbf{P_{0}}$  is crucial to this development.  The proof is in Section~\ref{sec: proofs}.

\begin{Prop} \label{pr: four}
$\mathbf{\pi_{s}}$  is continuous with respect to the total variation norm at $s=0$.  In particular, for $\epsilon > 0$  if we let $\sigma$  be the smallest nonzero singular value of $\mathbb{I} - \mathbf{P_{0}}$,  then if 
\begin{equation} \label{eq: ten}
\delta = \frac{ \epsilon \sigma}{2 n^{3 \slash 2}}
\end{equation}

\noindent we have for all $s \leq \delta$,  $\| \mathbf{\pi_{s}} - \mathbf{\pi_{0}} \|_{TV} \leq \epsilon$.
\end{Prop}

Now we can use Proposition \ref{pr: two}  along with the fact that $t_{mix}(\mathbf{P_{0}}, \epsilon) \leq t_{mix}(\epsilon)$  to derive the following Corollary to Proposition~\ref{pr: four}.

\begin{Cor} \label{cor: one}
$\mathbf{\pi_{s}}$  is continuous with respect to the total variation norm at $s=0$.  In particular, for $0 <\epsilon < 1 \slash \sqrt{n}$  if 
\begin{equation} \label{eq: eleven}
\delta = \frac{ \epsilon (1- \sqrt{n} \epsilon)}{4 n^{3 \slash 2} t_{mix}(\epsilon \slash 2)}
\end{equation}

\noindent we have for all $s \leq \delta$,  $\| \mathbf{\pi_{s}} - \mathbf{\pi_{0}} \|_{TV} \leq \epsilon \slash 2$.
\end{Cor}

We now have all the necessary tools to find a bound for the stable adiabatic time.  We find our result and conclude our paper in the following section.

\section{\sc {A Bound for the Stable Adiabatic Time}}
\label{sec: result}
We devote this section to finding a bound for the stable adiabatic time entirely in terms of the largest mixing time.  We state our main result in the following theorem.

\begin{Thm}
Given a time-inhomogeneous, discrete-time Markov chain governed by adiabatic evolution between the irreducible and aperiodic $\mathbf{P_{0}}$  and $\mathbf{P_{1}}$, for any $\epsilon > 0$,

\begin{equation}
t_{sad}(\mathbf{P_{0}}, \mathbf{P_{1}}, \epsilon) = O \left( \frac{t_{mix}^{4}(\epsilon \slash 2)}{\epsilon^{3}} \right).
\end{equation}

\end{Thm}

\begin{proof}
\

We first provide a sketch of the proof followed by the technical details.  
Our proof is based on the results of Theorem \ref{th: three}, Proposition \ref{pr: three}, and Corollary \ref{cor: one}.
Specifically, we divide our proof into two cases.  In the first case, we will show how to select $T$ and $\delta$ in order to satisfy the two conditions in Theorem \ref{th: three}, namely:

$$ T \geq \frac{2 t_{mix}^{2}(\epsilon \slash 2)}{\epsilon \delta}, $$
and 
$$ \delta \leq k \slash T \leq 1.$$  Therefore, by Theorem \ref{th: three},  we have
$$ \| \mathbf{\pi_{0}} \mathbf{P_{\frac{1}{T}}} \cdots \mathbf{P_{\frac{k}{T}}} - \mathbf{\pi_{\frac{k}{T}}} \|_{TV} \leq \epsilon. $$  However, the selected $T$ is not yet $t_{sad}(\mathbf{P_{0}}, \mathbf{P_{1}}, \epsilon)$ since this only holds for $k$ such that $$ \delta \leq k \slash T \leq 1.$$  In the second case, we will use the results of Proposition \ref{pr: three} and Corollary \ref{cor: one} to show that for the same selected $T$ and $\delta$,
$$ \| \mathbf{\pi_{0}} \mathbf{P_{\frac{1}{T}}} \cdots \mathbf{P_{\frac{k}{T}}} - \mathbf{\pi_{\frac{k}{T}}} \|_{TV} \leq \epsilon,$$ even in the case when $$k \slash T \leq \delta < 1.$$  Therefore, we conclude that the selected $T$ is a sufficient condition for $t_{sad}(\mathbf{P_{0}}, \mathbf{P_{1}}, \epsilon)$.

We now proceed with the details of the proof, starting with the first case.  
Let $\epsilon > 0$.  For this fixed $\epsilon$, we choose $T$  be an integer such that
\begin{align*}
T &\geq \frac{4 t_{mix}^{4}(\epsilon \slash 2)}{\epsilon^{3}} + \frac{4 t_{mix}^{2}(\epsilon \slash 2)}{\epsilon^{2}} + \frac{1}{\epsilon} \\
&= \left( \frac{2 t_{mix}^{2}(\epsilon \slash 2)}{\epsilon \sqrt{\epsilon}} + \frac{1}{\sqrt{\epsilon}} \right)^{2}.
\end{align*}

\noindent This implies
 \begin{align*}
 \sqrt{T} &\geq \frac{2 t_{mix}^{2}(\epsilon \slash 2)}{\epsilon \sqrt{\epsilon}} + \frac{1}{\sqrt{\epsilon}}.
 \end{align*}

\noindent Multiplying either side by $\sqrt{\epsilon}$  and subtracting $1$  from either side, we obtain
\begin{align*}
\sqrt{\epsilon} \sqrt{T} - 1 &\geq \frac{2 t_{mix}^{2}(\epsilon \slash 2)}{\epsilon}.
\end{align*}

\noindent Notice that  $\sqrt{\epsilon} \sqrt{T} - 1 > 0$  because $$\frac{2 t_{mix}^2(\epsilon \slash 2)}{\epsilon} > 0.$$

\noindent Dividing either side of the above inequality by $\sqrt{\epsilon} \sqrt{T} - 1$, we obtain
\begin{align*}
1 &\geq  \frac{2 t_{mix}^{2}(\epsilon \slash 2)}{\epsilon \left( \sqrt{\epsilon} \sqrt{T}  - 1 \right)}.
\end{align*}

\noindent Multiplying either side by $T$, we obtain
\begin{align*}
T &\geq \frac{2 t_{mix}^{2}(\epsilon \slash 2)}{\epsilon \left( \sqrt{\frac{\epsilon}{T}} - \frac{1}{T} \right)}.
\end{align*}

\noindent Now, let $$\delta =  \sqrt{\frac{\epsilon}{T}} - \frac{1}{T},$$

\noindent then clearly $$T \geq \frac{2 t_{mix}^{2}(\epsilon \slash 2)}{\epsilon \delta}.$$

\noindent Next, let $k$  be an integer such that $$ \delta = \sqrt{\frac{\epsilon}{T}} - \frac{1}{T} \leq \frac{k}{T} \leq 1.$$

\noindent Then by Theorem \ref{th: three}, we conclude that $$ \| \mathbf{\pi_{0}} \mathbf{P_{\frac{1}{T}}} \cdots \mathbf{P_{\frac{k}{T}}} - \mathbf{\pi_{\frac{k}{T}}} \|_{TV} \leq \epsilon.  $$

%\noindent for $$\sqrt{\frac{\epsilon}{T}} - \frac{1}{T} \leq \frac{k}{T} \leq 1. $$

Now in the second (complementary) case, i.e., when $k \slash T \leq \delta < 1$,  we will show that for the same selected $\delta =  \sqrt{\frac{\epsilon}{T}} - \frac{1}{T},$  and  $T$, it is still true that: $$ \| \mathbf{\pi_{0}} \mathbf{P_{\frac{1}{T}}} \cdots \mathbf{P_{\frac{k}{T}}} - \mathbf{\pi_{\frac{k}{T}}} \|_{TV} \leq \epsilon,$$

\noindent Let $k$  be an integer such that $$0 \leq \frac{k}{T} \leq \sqrt{\frac{\epsilon}{T}} - \frac{1}{T} = \delta.$$

\noindent Then, $$\frac{k+1}{T} \leq \sqrt{\frac{\epsilon}{T}}.$$

\noindent Using Proposition~\ref{pr: three}, we have

\begin{align*}
\| \mathbf{\pi_{0}} \mathbf{P_{\frac{1}{T}}} \cdots \mathbf{P_{\frac{k}{T}}} - \mathbf{\pi_{\frac{k}{T}}} \|_{TV} &\leq \| \mathbf{\pi_{\frac{k}{T}}} - \mathbf{\pi_{0}} \|_{TV} + \frac{(k+1)^{2}}{2 T} \\
&= \| \mathbf{\pi_{\frac{k}{T}}} - \mathbf{\pi_{0}} \|_{TV} + \frac{T}{2} \left( \frac{k+1}{T} \right)^{2} \\
&\leq \| \mathbf{\pi_{\frac{k}{T}}} - \mathbf{\pi_{0}} \|_{TV} + \frac{T}{2} \left( \sqrt{\frac{\epsilon}{T}} \right)^{2} \\
&= \| \mathbf{\pi_{\frac{k}{T}}} - \mathbf{\pi_{0}} \|_{TV} + \frac{\epsilon}{2}.
\end{align*}

\noindent Next, from Corollary~\ref{cor: one},  as long as $\epsilon < 1 \slash \sqrt{n}$  and $ \sqrt{\frac{\epsilon}{T}} - \frac{1}{T} \leq \frac{ \epsilon (1- \sqrt{n} \epsilon)}{4 n^{3 \slash 2} t_{mix}(\epsilon \slash 2)}$,

\noindent we have  $$ \| \mathbf{\pi_{0}} \mathbf{P_{\frac{1}{T}}} \cdots \mathbf{P_{\frac{k}{T}}} - \mathbf{\pi_{\frac{k}{T}}} \|_{TV} \leq \epsilon $$

\noindent for $$0 \leq \frac{k}{T} \leq \sqrt{\frac{\epsilon}{T}} - \frac{1}{T}.$$

\noindent It should be clear that as $\epsilon \rightarrow 0$, $$ \sqrt{\frac{\epsilon}{T}} - \frac{1}{T} \leq \frac{ \epsilon (1- \sqrt{n} \epsilon)}{4 n^{3 \slash 2} t_{mix}(\epsilon \slash 2)}  $$

\noindent when $$T \geq \frac{4 t_{mix}^{4}(\epsilon \slash 2)}{\epsilon^{3}} + \frac{4 t_{mix}^{2}(\epsilon \slash 2)}{\epsilon^{2}} + \frac{1}{\epsilon}.$$

\noindent This tells us that as $\epsilon \rightarrow 0$, $$t_{sad}(\mathbf{P_{0}}, \mathbf{P_{1}}, \epsilon) \leq \frac{4 t_{mix}^{4}(\epsilon \slash 2)}{\epsilon^{3}} + \frac{4 t_{mix}^{2}(\epsilon \slash 2)}{\epsilon^{2}} + \frac{1}{\epsilon}.$$

\noindent We conclude that $$ t_{sad}(\mathbf{P_{0}}, \mathbf{P_{1}}, \epsilon) = O \left( \frac{t_{mix}^{4}(\epsilon \slash 2)}{\epsilon^{3}} \right).$$

\end{proof}

We see that this result somewhat reaffirms what has been shown in the Quantum Adiabatic Theorem in \cite{ambainis2006elementary}, but a main difference is that the inverse spectral gap bound for the quantum system is replaced with a mixing time bound in our result.  Our result also has an extra multiple of $1 \slash \epsilon$.  Notice that the inverse spectral gap was a natural choice for the Quantum Adiabatic Theorem due to the Hamiltonian matrix being self-adjoint.  For general, not necessarily reversible, Markov Chains, the Adiabatic Theorem is expressed using mixing times. 

\section{\sc {Proofs}}
\label{sec: proofs}
\subsection{Proof of Proposition 1}
\ 

\noindent Recall the proof of Theorem \ref{th: one} in \cite{kovchegov2010note}.  We notice that $$t_{ad}(\mathbf{P_{0}}, \mathbf{P_{1}}, \epsilon) \leq K t_{mix}(\mathbf{P_{1}}, \epsilon \slash 2)$$

\noindent where $$ 1 + \left( \frac{\left(1 + \frac{1}{K-1} \right)^{K-1}}{e} \right)^{t_{mix}(\mathbf{P_{1}}, \epsilon \slash 2)} \leq \epsilon \slash 2. $$

\noindent After performing some basic algebra and taking the natural logarithm of either side of the equation, we see that
\begin{align*}
\ln \left( 1 - \epsilon \slash 2 \right) &\leq t_{mix} \left( \mathbf{P_{1}}, \epsilon \slash 2 \right) \left( \ln \left( \left( 1 + \frac{1}{K-1} \right)^{K-1} \right) - 1 \right) \\
&=  t_{mix} \left( \mathbf{P_{1}}, \epsilon \slash 2 \right) \left( \left( K - 1 \right) \ln \left( 1 + \frac{1}{K-1} \right) - 1 \right) \\
&= t_{mix} \left( \mathbf{P_{1}}, \epsilon \slash 2 \right) \left( \left( K - 1 \right)  \left( \sum_{j=1}^{\infty} (-1)^{j+1} \frac{1}{j(K-1)^{j} }\right) - 1 \right) \\
&= t_{mix} \left( \mathbf{P_{1}}, \epsilon \slash 2 \right) \left( \sum_{j=2}^{\infty} (-1)^{j+1} \frac{1}{j(K-1)^{j-1} } \right) \\
&= t_{mix} \left( \mathbf{P_{1}}, \epsilon \slash 2 \right) \left( \sum_{j=1}^{\infty} (-1)^{j+1} \frac{1}{j(K-1)^{j} } \left( \frac{-j}{j+1} \right) \right).
\end{align*}

\noindent It is clear now that if we select $K$  large enough so that 
\begin{align*} 
\ln \left( 1 - \epsilon \slash 2 \right) &\leq - t_{mix} \left( \mathbf{P_{1}}, \epsilon \slash 2 \right) \left( \sum_{j=1}^{\infty} (-1)^{j+1} \frac{1}{j(K-1)^{j} } \right) \\
&= - t_{mix} \left( \mathbf{P_{1}}, \epsilon \slash 2 \right) \ln \left( 1 + \frac{1}{K-1} \right)
\end{align*}

\noindent then $K$  will be large enough to satisfy the previous inequality. \\

\noindent Exponentiating either side of the equation and performing the basic algebra required to solve for $K$  we see that 
\begin{align*} 
K &\geq 1 + \left( e^{\left( \frac{- \ln (1 - \epsilon \slash 2)}{t_{mix}(\mathbf{P_{1}}, \epsilon \slash 2)} \right)} - 1 \right)^{-1} \\
&= 1 + \left( \sum_{j=0}^{\infty} \left( \frac{1}{j!} \left( \frac{ - \ln (1 - \epsilon \slash 2)}{t_{mix}(\mathbf{P_{1}}, \epsilon \slash 2) } \right)^{j} \right) - 1 \right)^{-1} \\
&= 1 + \left( \frac{- \ln (1 - \epsilon \slash 2)}{t_{mix}(\mathbf{P_{1}}, \epsilon \slash 2)} \sum_{j=1}^{\infty} \frac{1}{j!} \left( \frac{ - \ln (1 - \epsilon \slash 2)}{ t_{mix}(\mathbf{P_{1}}, \epsilon \slash 2)} \right)^{j-1} \right)^{-1} \\
&= 1 + \frac{t_{mix}(\mathbf{P_{1}}, \epsilon \slash 2)}{- \ln (1 - \epsilon \slash 2)} \left( \sum_{j=1}^{\infty} \frac{1}{j!} \left( \frac{ - \ln (1 - \epsilon \slash 2)}{ t_{mix}(\mathbf{P_{1}}, \epsilon \slash 2)} \right)^{j-1} \right)^{-1}. 
\end{align*} 

\noindent Notice that the infinite sum that we have is the sum of positive terms and the first term in the sum is $1$.  This tells us that $$ 1 \leq \sum_{j=1}^{\infty} \frac{1}{j!} \left( \frac{ - \ln (1 - \epsilon \slash 2)}{ t_{mix}(\mathbf{P_{1}}, \epsilon \slash 2)} \right)^{j-1} $$

\noindent therefore $$ 1 \geq \left( \sum_{j=1}^{\infty} \frac{1}{j!} \left( \frac{ - \ln (1 - \epsilon \slash 2)}{ t_{mix}(\mathbf{P_{1}}, \epsilon \slash 2)} \right)^{j-1} \right)^{-1}.  $$

\noindent This tells us that if we select $K$  such that $$ K \geq 1 + \frac{t_{mix}(\mathbf{P_{1}}, \epsilon \slash 2)}{- \ln ( 1 - \epsilon \slash 2)}$$

\noindent then the above inequality will be satisfied. \\

\noindent Finally we can expand $\ln (1 - \epsilon \slash 2)$  to find that $$K \geq 1 + \frac{2 t_{mix}(\mathbf{P_{1}}, \epsilon \slash 2)}{\epsilon} \left( \sum_{j=1}^{\infty} \frac{1}{j} \left( \frac{\epsilon}{2} \right)^{j-1} \right)^{-1}. $$

\noindent Again the infinite sum is the sum of positive terms, and the first term in the sum is $1$.  This tells us that $$ 1 \geq \left( \sum_{j=1}^{\infty} \frac{1}{j}  \left( \frac{\epsilon}{2} \right)^{j-1} \right)^{-1}.$$

\noindent We conclude that if we select $K$  such that $$K \geq \frac{2 t_{mix}(\mathbf{P_{1}}, \epsilon \slash 2)}{\epsilon}$$  then $$ 1 + \left( \frac{\left(1 + \frac{1}{K-1} \right)^{K-1}}{e} \right)^{t_{mix}(\mathbf{P_{1}}, \epsilon \slash 2)} \leq \epsilon \slash 2. $$

\noindent Therefore, we see that $$ t_{ad}(\mathbf{P_{0}}, \mathbf{P_{1}}, \epsilon) \leq \frac{2 t_{mix}^{2}(\mathbf{P_{1}}, \epsilon \slash 2)}{\epsilon}. $$

\subsection{Proof of Theorem 2}
\ 

\noindent To develop the tools for this theorem, we consider the following treatment of our probability transition matrices.  If we are given $s \in (0,1]$, then we see that $$\mathbf{P_{t}} = \left( 1 - \frac{t}{s} \right) \mathbf{P_{0}} + \frac{t}{s} \mathbf{P_{s}}$$  for all $t \in [0,s].$  

\noindent Defining $\mathbf{P_{t}^{(s)}} = \mathbf{P_{st}}$, we see that $$\mathbf{P_{t}^{(s)}} = (1 - t) \mathbf{P_{0}^{(s)}} + t \mathbf{P_{1}^{(s)}}$$  for all $t \in [0,1].$  We also define $\mathbf{\pi_{t}}^{(s)} = \mathbf{\pi_{st}}$.   

\noindent We see that $\{ \mathbf{P_{t}^{(s)}} \}_{t \in [0,1]}$  is a class of probability transition matrices where $\mathbf{P_{0}} = \mathbf{P_{0}^{(s)}}$  and $\mathbf{P_{s}} = \mathbf{P_{1}^{(s)}}$. 

\noindent Since the time-homogeneous Markov chains determined by $\mathbf{P_{0}}$  and $\mathbf{P_{s}}$  are irreducible and aperiodic, we can consider a time-inhomogeneous, discrete-time Markov chain governed by adiabatic evolution between these two time-homogeneous Markov chains.  We can apply Theorem \ref{th: one}  to show that $$ t_{ad}(\mathbf{P_{0}^{(s)}}, \mathbf{P_{1}^{(s)}}, \epsilon) = O \left( \frac{t_{mix}^{2}(\mathbf{P_{1}^{(s)}}, \epsilon \slash 2)}{\epsilon} \right). $$

\noindent Now let $\epsilon > 0$  and $\delta \in (0,1].$ \\

\noindent For $s \in [\delta, 1]$  we have that $T^{*} = t_{ad}( \mathbf{P_{0}^{(s)}}, \mathbf{P_{1}^{(s)}}, \epsilon)$  is the adiabatic time between $\mathbf{P_{0}^{(s)}}$  and $\mathbf{P_{1}^{(s)}}$. \\

\noindent This tells us that $$ \max_{\mathbf{\nu}} \| \mathbf{\nu} \mathbf{P_{\frac{1}{T^{*}}}^{(s)}} \mathbf{P_{\frac{2}{T^{*}}}^{(s)}} \cdots \mathbf{P_{1}^{(s)}} - \mathbf{\pi_{1}}^{(s)} \|_{TV} \leq \epsilon. $$

\noindent Because $\mathbf{\pi_{0}}^{(s)}$  is a specific distribution, we have that 

\begin{align*} 
\epsilon &\geq \| \mathbf{\pi_{0}}^{(s)} \mathbf{P_{\frac{1}{T^{*}}}^{(s)}} \mathbf{P_{\frac{2}{T^{*}}}^{(s)}} \cdots \mathbf{P_{1}^{(s)}} - \mathbf{\pi_{1}}^{(s)} \|_{TV} \\
&= \| \mathbf{\pi_{0}}^{(s)} \mathbf{P_{\frac{(1 \slash s)}{(T^{*} \slash s)}}^{(s)}} \mathbf{P_{\frac{( 2 \slash s )}{ ( T^{*} \slash s)}}^{(s)}} \cdots \mathbf{P_{\frac{(T^{*}\slash s)}{(T^{*} \slash s)}}^{(s)}} - \mathbf{\pi_{1}^{(s)}} \|_{TV} \\
&= \| \mathbf{\pi_{0}} \mathbf{P_{\frac{1}{(T^{*} \slash s)}}} \mathbf{P_{\frac{2}{ ( T^{*} \slash s)}}} \cdots \mathbf{P_{\frac{s(T^{*} \slash s)}{(T^{*} \slash s)}}} - \mathbf{\pi_{\frac{s(T^{*} \slash s)}{(T^{*} \slash s)}}} \|_{TV}.
\end{align*}

\noindent Clearly if $T = t_{ad}( \mathbf{P_{0}^{(s)}}, \mathbf{P_{1}^{(s)}}, \epsilon) \slash s$, then $$ \| \mathbf{\pi_{0}} \mathbf{P_{\frac{1}{T}}} \mathbf{P_{\frac{2}{T}}} \cdots \mathbf{P_{\frac{sT}{T}}} - \mathbf{\pi_{\frac{sT}{T}}} \|_{TV} \leq \epsilon.$$

\noindent We showed in Proposition \ref{pr: one} that for $\epsilon > 0$  $$ t_{ad}( \mathbf{P_{0}^{(s)}}, \mathbf{P_{1}^{(s)}}, \epsilon) \leq \frac{ 2 t_{mix}^{2}(\mathbf{P_{s}}, \epsilon \slash 2)}{\epsilon}.$$

\noindent It follows that for $\epsilon > 0$  $$ t_{ad}( \mathbf{P_{0}^{(s)}}, \mathbf{P_{1}^{(s)}}, \epsilon) \leq \frac{ 2 t_{mix}^{2}(\epsilon \slash 2)}{\epsilon}.$$

\noindent For $\epsilon > 0$  if we let $T$  any integer such that $$ T \geq \frac{ 2 t_{mix}^{2}(\epsilon \slash 2)}{\epsilon \delta}  $$

\noindent we have $$ \| \mathbf{\pi_{0}} \mathbf{P_{\frac{1}{T}}} \mathbf{P_{\frac{2}{T}}} \cdots \mathbf{P_{\frac{k}{T}}} - \mathbf{\pi_{\frac{k}{T}}} \|_{TV} \leq \epsilon$$

\noindent for all $\delta \leq k \slash T \leq 1$. \\

\subsection{Proof of Proposition 2}
\ 

\noindent We know that an irreducible, aperiodic time-homogeneous Markov chain governed by a probability transition matrix $\mathbf{P}$  has a unique stationary distribution, making the nullity of $(\mathbb{I} \lambda - \mathbf{P})$  equal to one when $\lambda = 1$.  This would necessarily imply that the rank of $(\mathbb{I} - \mathbf{P})$  is $n-1$.  

\noindent Let $\sigma_{1} \geq \cdots \geq \sigma_{n-1} = \sigma$  be the positive singular values of $( \mathbb{I} - \mathbf{P})$  with respect to the Euclidean inner product, which we will denote $\| \cdot \|_{2}$  throughout this paper.  This implies that there exists an orthonormal basis $\{ \mathbf{v_{1}}, \cdots, \mathbf{v_{n}} \}$  such that $\mathbf{v_{j}}(\mathbb{I} - \mathbf{P})(\mathbb{I} - \mathbf{P})^{T} = \sigma_{j}^{2} \mathbf{v_{j}}$  for $1 \leq j \leq n-1$  and $\mathbf{v_{n}}(\mathbb{I} - \mathbf{P})(\mathbb{I} - \mathbf{P})^{T} = \mathbf{0}$. 

\noindent Clearly $\mathbf{v_{n}} = \mathbf{\pi} \slash \| \mathbf{\pi} \|_{2}$.  \\

\noindent For $t \in \mathbb{N}$  define $\mathbf{M_{t-1}} = \mathbb{I} + \mathbf{P} + \mathbf{P}^{2} + \cdots + \mathbf{P}^{t-1}$. \\

\noindent Also define $\mathbf{\pi}$  to be the stationary distribution of $\mathbf{P}$. \\
 
\noindent Notice that $\mathbb{I} - \mathbf{P}^{t} = (\mathbb{I} - \mathbf{P}) \mathbf{M_{t-1}}$. \\

\noindent For irreducible, aperiodic Markov chains we have that if $\lambda_{1}, \cdots \lambda_{n}$  are the eigenvalues of $\mathbf{P}$  such that $1= \lambda_{1} > | \lambda_{2} | \geq \cdots \geq | \lambda_{n} |$, then $$t, \frac{1 - \lambda_{2}^{t}}{1 - \lambda_{2}}, \cdots, \frac{1 - \lambda_{n}^{t}}{1 - \lambda_{n}}$$

\noindent are the eigenvalues of $\mathbf{M_{t-1}}$. Notice that $\mathbf{M_{t-1}}$  must be invertible because all eigenvalues are nonzero and also notice that $t$  is the largest eigenvalue in modulus. \\

\noindent This implies that $\mathbb{I} - \mathbf{P} = (\mathbb{I} - \mathbf{P}^{t}) \mathbf{M_{t-1}}^{-1}$  and we see that $$ \sigma = \| \mathbf{v_{n-1}} (\mathbb{I} - \mathbf{P}) \|_{2} = \| \mathbf{v_{n-1}} (\mathbb{I} - \mathbf{P}^{t}) \mathbf{M_{t-1}}^{-1} \|_{2}. $$

\noindent We see that if $\| \cdot \|_{*}$  is the standard matrix norm, then

\begin{align*}
\| \mathbf{v_{n-1}} (\mathbb{I} - \mathbf{P}^{t}) \|_{2} &= \| \mathbf{v_{n-1}} (\mathbb{I} - \mathbf{P}^{t}) \mathbf{M_{t-1}}^{-1} \mathbf{M_{t-1}} \|_{2} \\
&\leq \| \mathbf{v_{n-1}} (\mathbb{I} - \mathbf{P}^{t}) \mathbf{M_{t-1}}^{-1} \|_{2} \| \mathbf{M_{t-1}} \|_{*} \\
&\leq t \| \mathbf{v_{n-1}} (\mathbb{I} - \mathbf{P}^{t}) \mathbf{M_{t-1}}^{-1} \|_{2} \\
&= t \sigma.
\end{align*}

\noindent If we let $\mathbf{u}$  be a vector such that for $1 \leq i \leq n$,  $\mathbf{u}(i) = 0$  whenever $\mathbf{v_{n-1}}(i) \geq 0$  and $\mathbf{u}(i) = - \mathbf{v_{n-1}}(i)$  whenever $\mathbf{v_{n-1}}(i) < 0$, then we have that $\mathbf{\nu_{1}} = \mathbf{u} \slash \| \mathbf{u} \|_{1}$  and $\mathbf{\nu_{2}} = (\mathbf{v_{n-1}} + \mathbf{u}) \slash \| \mathbf{v_{n-1}} + \mathbf{u} \|_{1}$  are probability distributions and 

\begin{align*} 
\mathbf{v_{n-1}}(\mathbb{I} - \mathbf{P}^{t}) =  \mathbf{v_{n-1}} &+ (\| \mathbf{u} \|_{1} - \| \mathbf{v_{n-1}} + \mathbf{u} \|_{1}) \mathbf{\pi} \\
&- \left( \| \mathbf{v_{n-1}} + \mathbf{u} \|_{1} \left( \mathbf{\nu_{2}} - \mathbf{\pi} \right) \mathbf{P}^{t} -  \| \mathbf{u} \|_{1} \left( \mathbf{\nu_{1}} - \mathbf{\pi} \right) \mathbf{P}^{t} \right).
\end{align*}

\noindent For $\mathbf{x},\mathbf{y} \in \mathbb{R}^{n}$ such that $\mathbf{x}$  and $\mathbf{y}$  are probability measures, we see that $$ \frac{1}{2} \| \mathbf{x} -\mathbf{y} \|_{2} \leq \| \mathbf{x} - \mathbf{y} \|_{TV} \leq \frac{\sqrt{n}}{2} \| \mathbf{x} - \mathbf{y} \|_{2}.$$

\noindent Through the triangle inequality we see that if we select $t = t_{mix}(\mathbf{P}, \epsilon)$, then 
\begin{align*}
\big\| \left( \| \mathbf{v_{n-1}} + \mathbf{u} \|_{1} \left( \mathbf{\nu_{2}} - \mathbf{\pi} \right) \mathbf{P}^{t} -  \| \mathbf{u} \|_{1} \left( \mathbf{\nu_{1}} - \mathbf{\pi} \right) \mathbf{P}^{t} \right) \big\|_{2} &\leq \| \mathbf{v_{n-1}} + \mathbf{u} \|_{1} \cdot \| \left( \mathbf{\nu_{2}} - \mathbf{\pi} \right) \mathbf{P}^{t} \|_{2}  \\
& \qquad +  \| \mathbf{u} \|_{1} \cdot \| \left( \mathbf{\nu_{1}} - \mathbf{\pi} \right) \mathbf{P}^{t} \|_{2} \\
&\leq 2 \| \mathbf{v_{n-1}} + \mathbf{u} \|_{1} \cdot \| \left( \mathbf{\nu_{2}} - \mathbf{\pi} \right) \mathbf{P}^{t} \|_{TV}  \\
& \qquad +  2 \| \mathbf{u} \|_{1} \cdot \| \left( \mathbf{\nu_{1}} - \mathbf{\pi} \right) \mathbf{P}^{t} \|_{TV} \\
&\leq 2 (\| \mathbf{v_{n-1}} + \mathbf{u}  \|_{1} + \| \mathbf{u} \|_{1}) \epsilon  \\
&= 2 \| \mathbf{v_{n-1}} \|_{1} \epsilon \\
&\leq 2 \sqrt{n} \| \mathbf{v_{n-1}} \|_{2} \epsilon \\
&= 2 \sqrt{n} \epsilon. 
\end{align*}

\noindent Because $\mathbf{v_{n-1}}$  and $\mathbf{\pi}$  are orthogonal, we see that 
\begin{align*} 
\big\|  \mathbf{v_{n-1}} + (\| \mathbf{u} \|_{1} - \| \mathbf{v_{n-1}} - \mathbf{u} \|_{1}) \mathbf{\pi} \big\|_{2} &= \sqrt{1 + (\| \mathbf{u} \|_{1} - \| \mathbf{v_{n-1}} + \mathbf{u} \|_{1})^{2} \left( \| \mathbf{\pi} \|_{2} \right)^{2}} \\
& \geq 1.
\end{align*}

\noindent Now through the reverse triangle inequality, meaning that for vectors $\mathbf{x}$  and $\mathbf{y}$, $ \| \mathbf{x} - \mathbf{y} \|_{2} \geq \| \mathbf{x} \|_{2} - \| \mathbf{y} \|_{2}$, we see that if $ t = t_{mix}(\mathbf{P}, \epsilon) $, then $$ \| \mathbf{v_{n-1}}(\mathbb{I} - \mathbf{P}^{t}) \|_{2} \geq 1 - 2 \sqrt{n} \epsilon.$$

\noindent This now implies that $$ \frac{1 - 2\sqrt{n} \epsilon}{\sigma} \leq t_{mix}(\mathbf{P}, \epsilon). $$

\subsection{Proof of Proposition 3}
\ 

\noindent Because $\mathbf{\pi_{0}} \mathbf{P_{\frac{j}{T}}} = \mathbf{\pi_{0}}  + \frac{j}{T} \mathbf{\pi_{0}} (\mathbf{P_{1}} - \mathbf{P_{0}})$  for $1 \leq j \leq k$  we notice that $$ \mathbf{\pi_{0}} \mathbf{P_{\frac{j}{T}}} \cdots \mathbf{P_{\frac{k}{T}}} - \mathbf{\pi_{\frac{k}{T}}} =\mathbf{\pi_{0}} \mathbf{P_{\frac{j+1}{T}}} \cdots \mathbf{P_{\frac{k}{T}}} - \mathbf{\pi_{\frac{k}{T}}} + \frac{j}{T} \mathbf{\pi_{0}} (\mathbf{P_{1}} - \mathbf{P_{0}})  \mathbf{P_{\frac{j+1}{T}}} \cdots \mathbf{P_{\frac{k}{T}}} $$

\noindent for $1 \leq j \leq k-1$  and $$ \mathbf{\pi_{0}} \mathbf{P_{\frac{k}{T}}} - \mathbf{\pi_{\frac{k}{T}}} = (\mathbf{\pi_{0}} - \mathbf{\pi_{\frac{k}{T}}}) + \frac{k}{T} \mathbf{\pi_{0}} (\mathbf{P_{1}} - \mathbf{P_{0}}).$$

\noindent Using the convention $\mathbf{P_{j+1}} \cdots \mathbf{P_{k}} = \mathbb{I}$  when $j \geq k$, we would see that $$ \mathbf{\pi_{0}} \mathbf{P_{\frac{1}{T}}} \cdots \mathbf{P_{\frac{k}{T}}} - \mathbf{\pi_{\frac{k}{T}}} = (\mathbf{\pi_{0}} - \mathbf{\pi_{\frac{k}{T}}}) + \sum_{j=1}^{k} \frac{j}{T} \mathbf{\pi_{0}} (\mathbf{P_{1}} - \mathbf{P_{0}}) \mathbf{P_{\frac{j+1}{T}}} \cdots \mathbf{P_{\frac{k}{T}}}.$$

\noindent Taking the total variation norm to either side of the inequality, using the triangle inequality  and pulling out constants, we see that
\begin{align*}
\| \mathbf{\pi_{0}} \mathbf{P_{\frac{1}{T}}} \cdots \mathbf{P_{\frac{k}{T}}} - \mathbf{\pi_{\frac{k}{T}}}  \|_{TV} &= \| (\mathbf{\pi_{0}} - \mathbf{\pi_{\frac{k}{T}}}) + \sum_{j=1}^{k} \frac{j}{T} \mathbf{\pi_{0}} (\mathbf{P_{1}} - \mathbf{P_{0}}) \mathbf{P_{\frac{j+1}{T}}} \cdots \mathbf{P_{\frac{k}{T}}} \|_{TV} \\
&\leq \| \mathbf{\pi_{0}} - \mathbf{\pi_{\frac{k}{T}}} \|_{TV} + \sum_{j=1}^{k} \frac{j}{T} \| \mathbf{\pi_{0}} (\mathbf{P_{1}} - \mathbf{P_{0}}) \mathbf{P_{\frac{j+1}{T}}} \cdots \mathbf{P_{\frac{k}{T}}} \|_{TV}.
\end{align*}

\noindent Notice that for $1 \leq j \leq k-1$  $$ \mathbf{\pi_{0}} (\mathbf{P_{1}} - \mathbf{P_{0}}) \mathbf{P_{\frac{j+1}{T}}} \cdots \mathbf{P_{\frac{k}{T}}} = \mathbf{\pi_{0}} \mathbf{P_{1}} \mathbf{P_{\frac{j+1}{T}}} \cdots \mathbf{P_{\frac{k}{T}}} - \mathbf{\pi_{0}} \mathbf{P_{0}} \mathbf{P_{\frac{j+1}{T}}} \cdots \mathbf{P_{\frac{k}{T}}} $$

\noindent is the difference between two probability distributions and $$ \mathbf{\pi_{0}}(\mathbf{P_{1}} - \mathbf{P_{0}})  = \mathbf{\pi_{0}} \mathbf{P_{1}} - \mathbf{\pi_{0}} \mathbf{P_{0}} $$

\noindent is also the difference between two probability distributions. \\

\noindent Because we are taking the total variation norm to the difference of two probability distributions we see that $ \|  \cdot \|_{TV} = \frac{1}{2} \| \cdot \|_{1}$  where $\|  \cdot \|_{1}$  is the $l_{1}$-norm. \\

\noindent We have that for probability distributions $\mathbf{\mu}$  and $\mathbf{\nu}$, $ \| \mathbf{\mu} - \mathbf{\nu} \|_{TV} = \frac{1}{2} \| \mathbf{\mu} - \mathbf{\nu} \|_{1} \leq \frac{1}{2} (\| \mathbf{\mu} \|_{1} + \| \mathbf{\nu} \|_{1}) \leq 1$. \\

\noindent This tells us that $\| \mathbf{\pi_{0}} (\mathbf{P_{1}} - \mathbf{P_{0}}) \mathbf{P_{\frac{j+1}{T}}} \cdots \mathbf{P_{\frac{k}{T}}} \|_{TV} \leq 1$  for $1 \leq j \leq k$. \\

\noindent We see then that 
\begin{align*}
\| \mathbf{\pi_{0}} \mathbf{P_{\frac{1}{T}}} \cdots \mathbf{P_{\frac{k}{T}}} - \mathbf{\pi_{\frac{k}{T}}}  \|_{TV} &\leq  \| \mathbf{\pi_{\frac{k}{T}}} - \mathbf{\pi_{0}} \|_{TV} + \sum_{j=1}^{k} \frac{j}{T} \\
&= \| \mathbf{\pi_{\frac{k}{T}}} - \mathbf{\pi_{0}}  \|_{TV} + \frac{1}{T} \sum_{j=1}^{k} j \\
&= \| \mathbf{\pi_{\frac{k}{T}}} - \mathbf{\pi_{0}} \|_{TV} + \frac{k(k+1)}{2T} \\
&\leq \| \mathbf{\pi_{\frac{k}{T}}} - \mathbf{\pi_{0}} \|_{TV} + \frac{(k+1)^{2}}{2T}.
\end{align*}

\subsection{Proof of Proposition 4}
\ 

\noindent We begin with the creation of an orthonormal basis of eigenvectors associated with $(\mathbb{I} - \mathbf{P_{0}})(\mathbb{I} - \mathbf{P_{0}})^{T}$  by a singular value decomposition similar to the process we mentioned in Proposition \ref{pr: two}.

\noindent Here we let $\sigma_{1} \geq \cdots \geq \sigma_{n-1} = \sigma$  be the positive singular values of $( \mathbb{I} - \mathbf{P_{0}})$  with respect to the Euclidean inner product.  This implies that there exists an orthonormal basis $\{ \mathbf{v_{1}}, \cdots, \mathbf{v_{n}} \}$  such that $\mathbf{v_{j}}(\mathbb{I} - \mathbf{P_{0}})(\mathbb{I} - \mathbf{P_{0}})^{T} = \sigma_{j}^{2} \mathbf{v_{j}}$  for $1 \leq j \leq n-1$  and $\mathbf{v_{n}}(\mathbb{I} - \mathbf{P_{0}})(\mathbb{I} - \mathbf{P_{0}})^{T} = \mathbf{0}$. \\

\noindent Here $\mathbf{v_{n}} = \mathbf{\pi_{0}} \slash \| \mathbf{\pi_{0}} \|_{2}$.  \\

\noindent To show continuity at  $s = 0$ let $\epsilon > 0$  and first notice that for any $s \in [0,1]$, $(\mathbf{\pi_{s}} - \mathbf{\pi_{0}})( \mathbb{I} - \mathbf{P_{0}} )= s \mathbf{\pi_{s}}( \mathbf{P_{1}} - \mathbf{P_{0}})$. \\

\noindent Using the Euclidean norm, we see that if $\mathbf{P_{0}} \neq \mathbf{P_{1}}$  and $s \neq 0$, then $$\frac{\| (\mathbf{\pi_{s}} - \mathbf{\pi_{0}})( \mathbb{I} - \mathbf{P_{0}} ) \|_{2}}{\| \mathbf{\pi_{s}} - \mathbf{\pi_{0}} \|_{2}}= s \frac{\| \mathbf{\pi_{s}}( \mathbf{P_{1}} - \mathbf{P_{0}}) \|_{2}}{ \| \mathbf{\pi_{s}} - \mathbf{\pi_{0}} \|_{2}}.$$

\noindent Throughout this proof we will use $< \cdot , \cdot >$  as the Euclidean inner product. \\

\noindent  For $1 \leq j \leq n$  let $c_{j} = <\mathbf{\pi_{s}} - \mathbf{\pi_{0}}, \mathbf{v_{j}}>$.  Then we see that $\mathbf{\pi_{s}} - \mathbf{\pi_{0}} = \sum_{j=1}^{n} c_{j} \mathbf{v_{j}}$.\\

\noindent We have that 
\begin{align*}
\frac{\| (\mathbf{\pi_{s}} - \mathbf{\pi_{0}})(\mathbb{I} - \mathbf{P_{0}}) \|_{2}^{2}}{\| \mathbf{\pi_{s}} - \mathbf{\pi_{0}} \|_{2}^{2}} &= \frac{<(\mathbf{\pi_{s}} - \mathbf{\pi_{0}})(\mathbb{I} - \mathbf{P_{0}}),(\mathbf{\pi_{s}} -\mathbf{\pi_{0}})(\mathbb{I} - \mathbf{P_{0}})>}{<\mathbf{\pi_{s}} - \mathbf{\pi_{0}},\mathbf{\pi_{s}} - \mathbf{\pi_{0}}>} \\
&= \frac{<\mathbf{\pi_{s}} - \mathbf{\pi_{0}}, (\mathbf{\pi_{s}} - \mathbf{\pi_{0}})(\mathbb{I} - \mathbf{P_{0}})(\mathbb{I} - \mathbf{P_{0}})^{T}>}{<\mathbf{\pi_{s}} - \mathbf{\pi_{0}},\mathbf{\pi_{s}} - \mathbf{\pi_{0}}>} \\
&= \frac{<\sum_{j=1}^{n} c_{j} \mathbf{v_{j}}, \sum_{j=1}^{n-1} \sigma_{j}^{2} c_{j} \mathbf{v_{j}}>}{<\sum_{j=1}^{n} c_{j} \mathbf{v_{j}},\sum_{j=1}^{n} c_{j} \mathbf{v_{j}}>} \\
&= \frac{\sum_{j=1}^{n-1} \sigma_{j}^{2} c_{j}^{2}}{\sum_{j=1}^{n} c_{j}^{2}} \\
&\geq \sigma_{n-1}^{2} \frac{\sum_{j=1}^{n-1} c_{j}^{2}}{\sum_{j=1}^{n} c_{j}^{2}} \\
&= \sigma_{n-1}^{2} \left( 1 -\frac{ c_{n}^{2}}{\sum_{j=1}^{n} c_{j}^{2}} \right) \\
&= \sigma_{n-1}^{2} \left( 1 - \left( \frac{<\mathbf{\pi_{s}} - \mathbf{\pi_{0}}, \mathbf{v_{n}}>}{\| \mathbf{\pi_{s}} - \mathbf{\pi_{0}} \|_{2}} \right)^{2} \right).
\end{align*}

\noindent If we let $\mathbf{w}(s) = (\mathbf{\pi_{s}} - \mathbf{\pi_{0}}) \slash \| \mathbf{\pi_{s}} - \mathbf{\pi_{0}} \|_{2}$  then we see that $$ \sigma_{n-1}^{2} \left( 1 - \left( <\mathbf{w}(s), \mathbf{v_{n}}> \right)^{2} \right) \leq s^{2} \frac{\| \mathbf{\pi_{s}}(\mathbf{P_{1}} - \mathbf{P_{0}}) \|_{2}^{2}}{\| \mathbf{\pi_{s}} - \mathbf{\pi_{0}} \|_{2}^{2}}.$$

\noindent Because $\mathbf{w}(s)$  and $\mathbf{v_{n}}$  are unit vectors, we can use the fact that $$\| \mathbf{w}(s) \|_{2}^{2} - 2<\mathbf{w}(s),\mathbf{v_{n}}> + \| \mathbf{v_{n}} \|_{2}^{2} = \| \mathbf{w}(s) - \mathbf{v_{n}} \|_{2}^{2}$$  

\noindent to show that $$1 - <\mathbf{w}(s), \mathbf{v_{n}}> = \frac{1}{2} \| \mathbf{w}(s) - \mathbf{v_{n}} \|_{2}^{2}$$  

\noindent and we can use the fact that $$\| \mathbf{w}(s) \|_{2}^{2} + 2<\mathbf{w}(s),\mathbf{v_{n}}> + \| \mathbf{v_{n}} \|_{2}^{2} = \| \mathbf{w}(s) + \mathbf{v_{n}} \|_{2}^{2}$$ 

\noindent to show that $$1 + <\mathbf{w}(s), \mathbf{v_{n}}> = \frac{1}{2} \| \mathbf{w}(s) + \mathbf{v_{n}} \|_{2}^{2}.$$

\noindent From this we see that $1 - \left( <\mathbf{w}(s), \mathbf{v_{n}}> \right)^{2} = \| \mathbf{w}(s) - \mathbf{v_{n}} \|_{2}^{2} \cdot \| \mathbf{w}(s) + \mathbf{v_{n}} \|_{2}^{2} \slash 4.$  Plugging this into our previous equation, we can see that $$ \frac{\sigma_{n-1}^{2}}{4} \| \mathbf{w}(s) - \mathbf{v_{n}} \|_{2}^{2} \cdot \| \mathbf{w}(s) + \mathbf{v_{n}} \|_{2}^{2} \leq s^{2} \frac{\| \mathbf{\pi_{s}}(\mathbf{P_{1}} - \mathbf{P_{0}}) \|_{2}^{2}}{\| \mathbf{\pi_{s}} - \mathbf{\pi_{0}} \|_{2}^{2}}.$$

\noindent After performing some basic algebra we see that $$ \| \mathbf{\pi_{s}} - \mathbf{\pi_{0}} \|_{2} \leq \frac{2s \| \mathbf{\pi_{s}} (\mathbf{P_{1}} - \mathbf{P_{0}}) \|_{2}}{\sigma_{n-1} \| \mathbf{w}(s) - \mathbf{v_{n}} \|_{2} \cdot \| \mathbf{w}(s) + \mathbf{v_{n}} \|_{2}} .$$

\noindent Notice that $<\mathbf{w}(s), \mathbf{1}> \slash \sqrt{n} = 0$  and $<\mathbf{v_{n}}, \mathbf{1}> \slash \sqrt{n} = 1 \slash \left( \sqrt{n} \| \mathbf{\pi_{0}} \|_{2} \right)$  for all $s \in [0,1]$.  Because these are the scalar components of the projections of $\mathbf{w}(s)$  and $\mathbf{v_{n}}$  onto $\mathbf{1}$  respectively, we see that the minimum possible value for $\| \mathbf{w}(s) - \mathbf{v_{n}} \|_{2}$  and $\| \mathbf{w}(s) + \mathbf{v_{n}} \|_{2}$  is at least $1 \slash \left( \sqrt{n} \| \mathbf{\pi_{0}} \|_{2} \right).$ \\

\noindent We now have that 
\begin{align*}
\| \mathbf{\pi_{s}} - \mathbf{\pi_{0}} \|_{2} &\leq \frac{2s n \| \mathbf{\pi_{0}} \|_{2}^{2} \cdot  \| \mathbf{\pi_{s}} (\mathbf{P_{1}} - \mathbf{P_{0}}) \|_{2}}{\sigma_{n-1}} \\
& \leq \frac{2s n \| \mathbf{\pi_{s}} (\mathbf{P_{1}} - \mathbf{P_{0}}) \|_{2}}{\sigma_{n-1}} \\
& = \frac{2s n \| \mathbf{\pi_{s}} (\mathbf{P_{1}} - \mathbf{P_{0}}) \|_{2}}{\sigma}.
\end{align*}

\noindent Again for $\mathbf{x},\mathbf{y} \in \mathbb{R}^{n}$ such that $\mathbf{x}$  and $\mathbf{y}$  are probability measures, we see that $$ \frac{1}{2} \| \mathbf{x} -\mathbf{y} \|_{2} \leq \| \mathbf{x} - \mathbf{y} \|_{TV} \leq \frac{\sqrt{n}}{2} \| \mathbf{x} - \mathbf{y} \|_{2}.$$

\noindent This will imply that $$\| \mathbf{\pi_{s}} - \mathbf{\pi_{0}} \|_{TV} \leq \frac{2s  n^{3 \slash 2} \| \mathbf{\pi_{s}} ( \mathbf{P_{1}} - \mathbf{P_{0}} )\|_{TV}}{\sigma_{n-1}}.$$

\noindent Because $\mathbf{\pi_{s}}(\mathbf{P_{1}} - \mathbf{P_{0}}) = \mathbf{\pi_{s}} \mathbf{P_{1}} - \mathbf{\pi_{s}} \mathbf{P_{0}}$  is the difference of two probability distributions, we see that $\|  \cdot \|_{TV} = \frac{1}{2} \| \cdot \|_{1}$  where $\| \cdot \|_{1}$  is the $l_{1}$-norm.  This implies that $$\| \mathbf{\pi_{s}}(\mathbf{P_{1}} - \mathbf{P_{0}}) \|_{TV} = \frac{1}{2} \| \mathbf{\pi_{s}} \mathbf{P_{1}} - \mathbf{\pi_{s}} \mathbf{P_{0}} \|_{1} \leq \frac{1}{2} \left( \| \mathbf{\pi_{s}} \mathbf{P_{1}} \|_{1} + \| \mathbf{\pi_{s}} \mathbf{P_{0}} \|_{1} \right) \leq 1.$$

\noindent This shows that $$\| \mathbf{\pi_{s}} - \mathbf{\pi_{0}} \|_{TV} \leq \frac{2s n^{3 \slash 2}}{\sigma}.$$

\noindent Clearly if $\epsilon > 0$, then $$s \leq \delta = \frac{ \epsilon \sigma}{2 n^{3 \slash 2}}$$

\noindent implies $\| \mathbf{\pi_{s}} - \mathbf{\pi_{0}} \|_{TV} \leq \epsilon$. \\

\noindent This shows that $\mathbf{\pi_{s}}$  is continuous at $s = 0$.

\end{document}